\documentclass[smallextended]{svjour3}

\smartqed

\usepackage{verbatim,a4wide}
\usepackage{amsmath}
\usepackage{amssymb}
\usepackage{amsthm}
\usepackage[a4paper,left=2.5cm,right=2.5cm,top=2cm,bottom=3cm]{geometry}
\usepackage{graphicx}
\usepackage{caption}
\usepackage{subfig}
\usepackage{grffile}
\usepackage{float}
\usepackage{bbm}
\usepackage{indentfirst}
\usepackage{booktabs}
\usepackage{multirow}
\usepackage{rotating}
\usepackage[T1,hyphens]{url}
\usepackage{hyperref}
\usepackage{epstopdf}

\usepackage[cp1250]{inputenc}
\usepackage[OT4]{fontenc}
\usepackage[english]{babel}

\numberwithin{equation}{section}
\numberwithin{theorem}{section}
\numberwithin{problem}{section}
\numberwithin{lemma}{section}
\numberwithin{corollary}{section}
\numberwithin{remark}{section}
\numberwithin{example}{section}
\numberwithin{proposition}{section}

\newtheorem{algorithm}{Algorithm}
\numberwithin{algorithm}{section}

\makeatletter
\let\c@proposition\c@theorem
\let\c@corollary\c@theorem
\let\c@remark\c@theorem
\let\c@problem\c@theorem
\let\c@lemma\c@theorem
\let\c@definition\c@theorem
\let\c@example\c@theorem
\makeatother

\newcommand{\ra}[1]{\renewcommand{\arraystretch}{#1}}

\newcommand{\R}{\mathbb R}

\newcommand{\dx}{\mbox{${\rm\,d}x$}}

\newcommand{\Ai}{{\rm Ai}}
\newcommand{\Bi}{{\rm Bi}}

\begin{document}

%%%%%%%%%%
\title{An iterative approximate method of solving boundary value problems using dual Bernstein polynomials}
%%%%%%%%%%

\author{Przemys{\l}aw Gospodarczyk        \and
        Pawe{\l} Wo\'{z}ny
}

\institute{P. Gospodarczyk (Corresponding author) \at Institute of Computer Science, University of Wroc{\l}aw,
              ul.~F.~Joliot-Curie 15, 50-383 Wroc{\l}aw, Poland\\
              Fax: +48713757801\\
              \email pgo@ii.uni.wroc.pl
           \and P. Wo\'{z}ny \at Institute of Computer Science, University of Wroc{\l}aw,
                ul.~F.~Joliot-Curie 15, 50-383 Wroc{\l}aw, Poland\\
                \email Pawel.Wozny@cs.uni.wroc.pl
}

\date{\today}

\maketitle

\begin{abstract}
In this paper, we present a new iterative approximate method of solving boundary value problems. The idea is to compute approximate polynomial
solutions in the Bernstein form using least squares approximation combined with some properties of dual Bernstein polynomials which guarantee
high efficiency of our approach. The method can deal with both linear and nonlinear differential equations. Moreover, not only second order
differential equations can be solved but also higher order differential equations. Illustrative examples confirm the versatility of our method.


\end{abstract}

%%%%%%%%%%%%%%%%%%%%%%%%%%%%%%%%%%%%%%%%%%%%%%%%%%%%%%%%%%%%%%%%%%%%%%%%%%%%%%
%%%%%%%%%%%%%%%%%%%%%%%%%%%%%%%%%%%%%%%%%%%%%%%%%%%%%%%%%%%%%%%%%%%%%%%%%%%%%%
\section{Introduction}\label{Sec:Intro}
%%%%%%%%%%%%%%%%%%%%%%%%%%%%%%%%%%%%%%%%%%%%%%%%%%%%%%%%%%%%%%%%%%%%%%%%%%%%%%
%%%%%%%%%%%%%%%%%%%%%%%%%%%%%%%%%%%%%%%%%%%%%%%%%%%%%%%%%%%%%%%%%%%%%%%%%%%%%%

In the paper, we consider the following form of the boundary value problem.

\begin{problem} \,[\textsf{Boundary value problem}]\label{P:BVP}
Solve the $m$th order differential equation
\begin{equation}\label{E:diff}
y^{(m)}(x) = f\left(x,y(x),y'(x),\ldots,y^{(m-1)}(x)\right) \qquad (0 \le x \le 1)
\end{equation}
with the boundary conditions
\begin{align}
& y^{(i)}(0) = a_i \qquad (i=0,1,\ldots,k-1),\label{E:gencond1}\\
& y^{(j)}(1)=b_j \qquad (j=0,1,\ldots,l-1),\label{E:gencond2}
\end{align}
where $a_i, b_j \in \R$, and $k+l=m$. The differential equation \eqref{E:diff} may be nonlinear.
\end{problem}

The goal of the paper is to present a new iterative approximate method of solving Problem \ref{P:BVP}. The method is based on least squares approximation,
and complies with the requirements given below.
\begin{enumerate}
\item An approximate polynomial solution in the Bernstein form is computed. Therefore, the solution is given in the whole interval
$[0,\,1]$, not only at certain points.
\item The degree of the approximate polynomial solution can be chosen arbitrarily.
\item The method deals with both linear and nonlinear differential equations \eqref{E:diff}.
\item Not only second order differential equations \eqref{E:diff} can be solved but also higher order differential equations \eqref{E:diff}.
\item High efficiency is achieved thanks to some properties of Bernstein and dual Bernstein polynomials.
\end{enumerate}

As we shall see, the Bernstein form is a very suitable choice for an approximate polynomial solution of Problem \ref{P:BVP}. Here are the most important reasons.
\begin{enumerate}
\item The boundary conditions \eqref{E:gencond1} and \eqref{E:gencond2} can be easily satisfied by a polynomial in the Bernstein form.
\item Due to some properties of dual Bernstein polynomials, the cost of an iteration of our algorithm is $O(n^2)$, where $n$ is the degree
of the approximate polynomial solution in the Bernstein form which is derived in the iteration. The iteration is based, among other things,
on solving a least squares approximation problem. Note that the standard method of dealing with the least squares approximation is to solve
the system of normal equations with the complexity $O(n^3)$.
\item Bernstein polynomials can be easily differentiated and integrated.
\end{enumerate}

In recent years, the idea of representing approximate solutions of differential equations in the Bernstein basis has been used in several papers
(see, e.g., \cite{BB06,BB07,TS17}). The least squares approximation and orthogonality have been used for many years usually
in conjunction with \textit{Chebyshev polynomials} and \textit{Picard's iteration} (see, e.g., \cite{CN63,Fuk97,JYWB13,Wri64}).
However, according to the authors' knowledge, there are no methods based on the idea presented in this paper.

The paper is organized as follows. Section \ref{S:Prelim} has a preliminary character. In Section \ref{S:Main}, we give a description of our
method including algorithmic details of the implementation. Several illustrative examples are presented in Section \ref{S:Ex}.
For a brief summary of the paper, see Section \ref{S:Co}.

%%%%%%%%%%%%%%%%%%%%%%%%%%%%%%%%%%%%%%%%%%%%%%%%%%%%%%%%%%%%%%%%%%%%%%%%%%%%%%
%%%%%%%%%%%%%%%%%%%%%%%%%%%%%%%%%%%%%%%%%%%%%%%%%%%%%%%%%%%%%%%%%%%%%%%%%%%%%%
\section{Preliminaries}\label{S:Prelim}
%%%%%%%%%%%%%%%%%%%%%%%%%%%%%%%%%%%%%%%%%%%%%%%%%%%%%%%%%%%%%%%%%%%%%%%%%%%%%%
%%%%%%%%%%%%%%%%%%%%%%%%%%%%%%%%%%%%%%%%%%%%%%%%%%%%%%%%%%%%%%%%%%%%%%%%%%%%%%

Let $\Pi_n$ denote the space of all polynomials of degree at most $n$. \textit{Bernstein polynomials of degree $n$} are defined by
\begin{equation}\label{E:Bern}
B_i^n(x) = \binom{n}{i}x^i(1-x)^{n-i} \qquad (i=0,1,\ldots,n),
\end{equation}
and form a basis of $\Pi_n$. These polynomials are known for their important applications in \textit{computer-aided design}
and \textit{approximation theory} (see, e.g., \cite{Far12}). In each iteration, our method computes coefficients $p_{i,n}$ $(i=0,1,\ldots,n)$
of the \textit{Bernstein form} of an approximate solution
\begin{equation}\label{E:wnn}
w_n(x) = \sum_{i=0}^np_{i,n}B_i^n(x)
\end{equation}
of Problem \ref{P:BVP}.

Now, let us deal with the boundary conditions \eqref{E:gencond1} and \eqref{E:gencond2}. This can be done easily thanks to some properties of the
Bernstein polynomials \eqref{E:Bern}. See Lemmas \ref{L:der} and \ref{L:boundary}.

\begin{lemma}[\protect{e.g., \cite[\S5.3]{Far02}}]\label{L:der}
The $r$th derivative of the polynomial \eqref{E:wnn} is given by the formula
\begin{equation*}
w^{(r)}_n(x) = \sum_{j=0}^{n-r}p^{(r)}_{j,n}B_j^{n-r}(x),
\end{equation*}
where
\begin{equation}\label{E:derCoeff}
p^{(r)}_{j,n} := \frac{n!}{(n-r)!}\Delta^{r}p_{j,n} \qquad (j=0,1,\ldots,n-r)
\end{equation}
with the \emph{forward difference operator} $\Delta^r$ that acts on the first index (second index being fixed),
\begin{equation}\label{E:del}
\Delta^rp_{j,n} := \sum_{h=0}^{r} (-1)^{r-h}\binom{r}{h}p_{j+h,n}.
\end{equation}
Moreover,
\begin{align}
& w_n^{(r)}(0) = \frac{n!}{(n-r)!}\sum_{h=0}^{r}(-1)^{r-h}\binom{r}{h}p_{h,n},\label{E:wnder0}\\[1ex]
& w_n^{(r)}(1) = \frac{n!}{(n-r)!}\sum_{h=0}^{r}(-1)^{r-h}\binom{r}{h}p_{n-r+h,n}.\label{E:wnder1}
\end{align}
\end{lemma}

\begin{lemma}\label{L:boundary}
The polynomial \eqref{E:wnn}, which has the \emph{outer} coefficients
\begin{align}
& p_{i,n} = \frac{(n-i)!}{n!}a_i - \sum_{h=0}^{i-1}(-1)^{i-h}\binom{i}{h}p_{h,n} \qquad (i=0,1,\ldots,k-1),\label{E:p1}\\
& p_{n-j,n} = (-1)^j\frac{(n-j)!}{n!}b_j - \sum_{h=1}^{j}(-1)^{h}\binom{j}{h}p_{n-j+h,n} \qquad (j=0,1,\ldots,l-1),\label{E:p2}
\end{align}
and arbitrary \emph{inner} coefficients $p_{i,n}$ $(i=k,k+1,\ldots,n-l)$, satisfies the boundary conditions
\begin{align}
& w_n^{(i)}(0) = a_i \qquad (i=0,1,\ldots,k-1),\label{E:wncon0}\\
& w_n^{(j)}(1)=b_j \qquad (j=0,1,\ldots,l-1)\label{E:wncon1}
\end{align}
(cf.~\eqref{E:gencond1} and \eqref{E:gencond2}).
\end{lemma}
\begin{proof}
We apply the formulas \eqref{E:wnder0} and \eqref{E:wnder1} to \eqref{E:wncon0} and \eqref{E:wncon1}, respectively. After some algebra,
we derive \eqref{E:p1} and \eqref{E:p2}.
\end{proof}

\begin{remark}
As is known, the boundary value problem with nonhomogeneous boundary conditions \eqref{E:gencond1} and \eqref{E:gencond2} (see Problem \ref{P:BVP})
can always be converted to the boundary value problem with homogeneous boundary conditions, i.e., the conditions with $a_i=0$ $(i=0,1,\ldots,k-1)$
and $b_j=0$ $(j=0,1,\ldots,l-1)$ (cf.~\eqref{E:wncon0} and \eqref{E:wncon1}). To satisfy these homogeneous boundary conditions, no computations are
required since the outer coefficients of a polynomial solution in the Bernstein form are equal to $0$ (cf.~\eqref{E:p1} and \eqref{E:p2}).
This is one of the reasons why the Bernstein basis is so suitable for boundary value problems.
\end{remark}

According to Lemma \ref{L:boundary}, the outer coefficients $p_{i,n}$ $(i=0,1,\ldots,k-1)$ and $p_{n-j,n}$ $(j=0,1,\ldots,l-1)$ depend only
on the boundary conditions \eqref{E:gencond1} and \eqref{E:gencond2}. As a result, these coefficients can be easily computed at the beginning of
each iteration of the algorithm. What remains is to compute efficiently the optimal inner coefficients $p_{i,n}$ $(i=k,k+1,\ldots,n-l)$.
As we shall see, dual Bernstein polynomials are of great importance in this approximation process.

Let us define the \textit{$L_2$ inner product} $\langle\cdot,\cdot\rangle$ by
\begin{equation}\label{Eq:inprod}
\langle f,g\rangle := \int_0^1f(x)g(x)\dx.
\end{equation}
The \textit{dual Bernstein polynomial basis of degree $n$} (see \cite{Cie87})
\begin{equation}\label{E:dual}
D_{0}^n(x),D_1^n(x),\ldots, D_n^n(x)
\end{equation}
satisfies the relation
$$
\left\langle D_i^n,B_j^n\right\rangle =\delta_{i,j} \qquad (i,j=0,1,\ldots, n),
$$
where $\delta_{i,j}$ equals $1$ if $i=j$, and $0$ otherwise. According to the next lemma, dual Bernstein polynomials \eqref{E:dual}
can be efficiently represented in the Bernstein basis.

\begin{lemma}[\cite{LW11}]\label{L:c-tab}
Dual Bernstein polynomials \eqref{E:dual} have the Bernstein representation
\begin{equation*}
	 D^{n}_i(x)=\sum_{j=0}^{n}c_{i,j}^{(n)}B^{n}_j(x) \qquad (i=0,1,\ldots,n),
\end{equation*}
where the coefficients $c_{i,j}^{(n)}$ satisfy the recurrence relation
\begin{align*}\label{E:c-rec}
&c_{i+1,j}^{(n)} = \frac{1}{A(i)}\left[2(i-j)(i+j-n)c_{i,j}^{(n)}+B(j)c_{i,j-1}^{(n)}+A(j)c_{i,j+1}^{(n)}-B(i)c_{i-1,j}^{(n)}\right]\\ \notag
&\hphantom{c_{i+1,j}^{(n)} = \frac{1}{A(i)}\left[2(i-j)(i+j-n)c_{i,j}^{(n)}+B(j)c_{i,j-1}^{(n)}\right.}(i=0,1,\ldots,n-1;\;j=0,1,\ldots,n)
\end{align*}	
with
\[
\begin{array}{l}
A(u):= (u-n)(u+1),\\[1ex]	
B(u):= u(u-n-1).
\end{array}
\]
We adopt the convention that $c_{i,j}^{(n)} := 0$ if $i < 0$, or $j < 0$, or $j > n$.
The starting values are
\begin{equation*}
c_{0,j}^{(n)} := (-1)^j \frac{(n+1)(n+1-j)_{j+1}}{(j+1)!} \qquad (j=0,1,\ldots,n).
\end{equation*}
\end{lemma}
\noindent Notice that Lemma \ref{L:c-tab} results in an efficient algorithm of computing the connection coefficients $c_{i,j}^{(n)}$ $(i,j=0,1,\ldots,n)$
with the complexity $O(n^2)$ (cf.~\cite[Algorithm 3.3 for $\alpha,\beta,k,l=0$]{LW11}).

Recently, dual Bernstein polynomials have been extensively studied because of their applications in numerical analysis and computer-aided design
(see, e.g., \cite{GLW15,GLW17,LW11,WGL15,WL09}). In this paper, we present another application of these dual bases.

%%%%%%%%%%%%%%%%%%%%%%%%%%%%%%%%%%%%%%%%%%%%%%%%%%%%%%%%%%%%%%%%%%%%%%%%%%%%%%
%%%%%%%%%%%%%%%%%%%%%%%%%%%%%%%%%%%%%%%%%%%%%%%%%%%%%%%%%%%%%%%%%%%%%%%%%%%%%%
\section{An iterative approximate method of solving boundary value problems using dual Bernstein polynomials}\label{S:Main}
%%%%%%%%%%%%%%%%%%%%%%%%%%%%%%%%%%%%%%%%%%%%%%%%%%%%%%%%%%%%%%%%%%%%%%%%%%%%%%
%%%%%%%%%%%%%%%%%%%%%%%%%%%%%%%%%%%%%%%%%%%%%%%%%%%%%%%%%%%%%%%%%%%%%%%%%%%%%%

In this section, we first present a sketch of the idea for our iterative approximate method of solving Problem \ref{P:BVP}. Then, we give algorithmic
details of the implementation.

Let us assume that we are looking for the approximate polynomial solution of degree $N$ $(N \ge m)$,
\begin{equation}\label{E:wNtarget}
w_{N}(x) := \sum_{i=0}^{N}p_{i,N}B_i^{N}(x),
\end{equation}
i.e., the goal is to compute the coefficients $p_{i,N}$ $(i=0,1,\ldots,N)$. Recall that $m$ is the order of the differential equation \eqref{E:diff}, and
$m=k+l$.

We begin with the polynomial of degree $m-1$,
\begin{equation*}
w_{m-1}(x) := \sum_{i=0}^{m-1}p_{i,m-1}B_i^{m-1}(x),
\end{equation*}
where the coefficients $p_{i,m-1}$ $(i=0,1,\ldots,m-1)$ are given by the formulas \eqref{E:p1} and \eqref{E:p2} for $n := m-1$.
Therefore, $w_{m-1}$ depends only on the boundary conditions \eqref{E:gencond1} and \eqref{E:gencond2}, and the least
squares approximation is not involved here.

Next, for $n=m,m+1,\ldots,N$, the goal is to compute the approximate solution $w_n \in \Pi_n$ assuming
that the approximate solution $w_{n-1} \in \Pi_{n-1}$ was computed in the previous iteration of the algorithm.
The approximate solution
\begin{equation}\label{E:wn}
w_n(x) = \sum_{i=0}^np_{i,n}B_i^n(x)
\end{equation}
must satisfy
\begin{itemize}
\item[(i)] the \textit{$L_2$ optimality condition}
\begin{equation}\label{E:L2task}
\left\|w^{(m)}_n - f\left(\cdot,w_{n-1},w'_{n-1},\ldots,w_{n-1}^{(m-1)}\right)\right\| = \min_{w \in \Pi_n} \left\|w^{(m)} - f\left(\cdot,w_{n-1},w'_{n-1},\ldots,w_{n-1}^{(m-1)}\right) \right\|,
\end{equation}
where $\|\cdot\|:=\sqrt{\left\langle\cdot,\cdot\right\rangle}$ is the $L_2$ norm (cf.~\eqref{Eq:inprod}),
\item[(ii)] the boundary conditions
\begin{align}
& w_n^{(i)}(0) = a_i \qquad (i=0,1,\ldots,k-1),\label{E:wcond1}\\
& w_n^{(j)}(1)= b_j \qquad (j=0,1,\ldots,l-1)\label{E:wcond2}
\end{align}
(cf.~\eqref{E:gencond1} and \eqref{E:gencond2}).
\end{itemize}
According to Lemma \ref{L:boundary}, the outer coefficients $p_{i,n}$ $(i=0,1,\ldots,k-1)$ and $p_{n-j,n}$ $(j=0,1,\ldots,l-1)$ are given by \eqref{E:p1}
and \eqref{E:p2}, respectively. Now, the boundary conditions \eqref{E:wcond1} and \eqref{E:wcond2} are satisfied, and the goal is to compute the inner
coefficients $p_{i,n}$ $(i=k,k+1,\ldots,n-l)$ so that the $L_2$ optimality condition \eqref{E:L2task} is satisfied. Theorem \ref{T:main} shows how to deal
with this problem efficiently.

\begin{theorem}\label{T:main}
The inner coefficients $p_{i,n}$ $(i=k,k+1,\ldots,n-l)$ of the approximate solution $w_n$ (see \eqref{E:wn}) satisfy the
\emph{Toeplitz system of linear equations}
\begin{equation}\label{E:sys}
\mathbf{G}_n\mathbf{p}_n=\mathbf{v}_n,
\end{equation}
where $\mathbf{G}_n:=\left[g_{i,j}^{(n)}\right] \in \R^{(n-m+1)\times(n-m+1)}$, $\mathbf{p}_n := \left[p_{k,n},p_{k+1,n},\ldots,p_{n-l,n}\right]^T$,
$\mathbf{v}_n := \left[v_i^{(n)}\right] \in \R^{n-m+1}$,
\begin{align}
&g_{i,j}^{(n)} := (-1)^{l+i-j}\binom{m}{j+k-i} \qquad (i,j=0,1,\ldots,n-m),\label{E:gij}\\[1ex]
\notag&v_i^{(n)} := \frac{(n-m)!}{n!}\sum_{q=0}^{n-m}c_{i,q}^{(n-m)}I_{q,n}
-\left(\sum_{h=0}^{k-i-1}+\sum_{h=n-l-i+1}^{m}\right)(-1)^{m-h}\binom{m}{h}p_{i+h,n}\\
\notag&\hphantom{v_i^{(n)} := \frac{(n-m)!}{n!}\sum_{q=0}^{n-m}c_{i,q}^{(n-m)}I_{q,n}
-\left(\sum_{h=0}^{k-i-1}+\sum_{h=n-l-i+1}^{m}\right)(-1)^{m-h}} (i=0,1,\ldots,n-m)
\end{align}
with $c_{i,j}^{(n-m)}$ as defined in Lemma \ref{L:c-tab}, and
\begin{equation}\label{E:int}
I_{q,n} := \left\langle f\left(\cdot,w_{n-1},w'_{n-1},\ldots,w_{n-1}^{(m-1)}\right),B_q^{n-m}\right\rangle \qquad (q=0,1,\ldots,n-m).
\end{equation}
\end{theorem}
\begin{proof}
First, using Lemma \ref{L:der}, we obtain
\begin{equation*}
w^{(m)}_n(x) = \sum_{i=0}^{n-m}p^{(m)}_{i,n}B_i^{n-m}(x),
\end{equation*}
where
\begin{equation}\label{E:wnm1}
p^{(m)}_{i,n} := \frac{n!}{(n-m)!}\sum_{h=0}^{m} (-1)^{m-h}\binom{m}{h}p_{i+h,n} \qquad (i=0,1,\ldots,n-m).
\end{equation}
Next, $w^{(m)}_n$ must satisfy the $L_2$ optimality condition \eqref{E:L2task}. Remembering that $B_i^{n-m}$ and $D_i^{n-m}$ $(i=0,1,\ldots,n-m)$ are
dual bases of the space $\Pi_{n-m}$, and using the Bernstein representation of $D_i^{n-m}$ (see Lemma \ref{L:c-tab}), we derive the formulas for the
optimal values of the coefficients $p^{(m)}_{i,n}$,
\begin{align}
\notag p^{(m)}_{i,n} &:= \left\langle f\left(\cdot,w_{n-1},w'_{n-1},\ldots,w_{n-1}^{(m-1)}\right),D_i^{n-m}\right\rangle\\
              &= \sum_{j=0}^{n-m}c_{i,j}^{(n-m)}\left\langle f\left(\cdot,w_{n-1},w'_{n-1},\ldots,w_{n-1}^{(m-1)}\right),B_j^{n-m}\right\rangle
             \qquad (i=0,1,\ldots,n-m).\label{E:wnm2}
\end{align}
Now, we equate \eqref{E:wnm1} to \eqref{E:wnm2}, and obtain the system \eqref{E:sys}. Finally, since $\binom{p}{r} = 0$ if $r < 0$ or $r > p$, it can be
easily checked that $g_{i,j}^{(n)} = 0$ if $i-j>k$ or $j-i>l$ (see \eqref{E:gij}). Therefore, \eqref{E:sys} is a Toeplitz system of linear equations.
\end{proof}

\begin{remark}\label{R:Toe}
Since \eqref{E:sys} is a Toeplitz system of linear equations, it can be solved, in general, with the complexity $O(n^2)$ using
\textit{generalized Levinson's algorithm} (see, e.g., \cite[\S2.8]{PTVF07}). There are also algorithms that solve Toeplitz systems of linear
equations with the complexity $O(n\log^2n)$ (see, e.g., \cite{Bun85,Hoog87} and the lists of references given there), but in the context of our
problem their significance is only theoretical. Moreover, since $\mathbf{G}_n$ is associated with the forward difference operator (see \eqref{E:del}
and \eqref{E:gij}), in some cases it can be inverted using explicit formulas (see, e.g., \cite{All73,HP72}).

Now, let us assume that $k, l \ll n$ which is the most common case for our problem. The bandwidth of $\mathbf{G}_n$ is thus very small.
Consequently, the best option in practice is to solve the system \eqref{E:sys} using \textit{Gaussian elimination for band matrices} with the complexity
$O(kln)$ (see, e.g., \cite[\S4.3]{GL96}). Furthermore, let us list some special cases that can be treated separately in order to speed up the computations.
\begin{enumerate}
\item If $k=0$, then $\mathbf{G}_n$ is an \textit{upper triangular matrix}, and the system \eqref{E:sys} can be solved using  \textit{back substitution}
with the complexity $O(ln)$ (see, e.g., \cite[\S4.3]{GL96}).
\item If $l=0$, then $\mathbf{G}_n$ is a \textit{lower triangular matrix}, and the system \eqref{E:sys} can be solved using  \textit{forward substitution}
 with the complexity $O(kn)$ (see, e.g., \cite[\S4.3]{GL96}).
\item If $k=l=1$, then $\mathbf{G}_n$ is a \textit{tridiagonal matrix}, and the system \eqref{E:sys} can be solved with the complexity $O(n)$
(see, e.g., \cite[\S2.4]{PTVF07}).
\end{enumerate}
\end{remark}

\begin{remark}\label{R:der1}
Notice that the computation of \eqref{E:int} requires a method of computing the collection of integrals
$$
\int_0^1 f(x)x^q(1-x)^{n-m-q}\dx \qquad (q=0,1,\ldots,n-m),
$$
where $f(x) \equiv f\left(x,w_{n-1}(x),w'_{n-1}(x),\ldots,w_{n-1}^{(m-1)}(x)\right)$.
Since we do not assume anything about $f$, it is impossible to recommend one specific method working for every example.
However, \textit{Gaussian quadratures} (see, e.g., \cite[\S4.6]{PTVF07}) or well-developed algorithms provided by some computing
environments may be good choices for many examples. We used Maple{\small \texttrademark}13 function \texttt{int} with the
option \texttt{numeric}, and the results look fine (see Section \ref{S:Ex}).
\end{remark}

\begin{remark}\label{R:der2}
Observe that $f$ in \eqref{E:int} requires $w_{n-1}^{(r)}$ $(r=1,2,\ldots,m-1)$. According to Lemma \ref{L:der}, these polynomials can be represented
in the Bernstein form. The coefficients $p_{j,n-1}^{(r)}$ $(r=1,2,\ldots,m-1;\;j=0,1,\ldots,n-r-1)$ of these representations can be computed
using \eqref{E:derCoeff} with \eqref{E:del}. However, this approach is inefficient, i.e., the complexity is $O(nm^2)$. Note that it is much more efficient to
put $\Delta^rp_{j,n-1}$ $(r=1,2,\ldots,m-1;\;j=0,1,\ldots,n-r-1)$ in a table, and complete this table using the well-know recurrence relation for
the forward difference operator,
\begin{equation*}
\Delta^rp_{j,n-1} = \Delta^{r-1}p_{j+1,n-1} - \Delta^{r-1}p_{j,n-1} \qquad (r=1,2,\ldots,m-1;\;j=0,1,\ldots,n-r-1),
\end{equation*}
where $\Delta^{0}p_{q,n-1} = p_{q,n-1}$ $(q=0,1,\ldots,n-1)$. The complexity of this approach is $O(nm)$. Moreover, the application of a
quadrature to \eqref{E:int} (see Remark \ref{R:der1}) requires a method of evaluating a polynomial in the Bernstein form.
One can use \textit{Horner-like scheme} (see, e.g., \cite[\S5.4.2]{Gol03}) which has the complexity $O(n)$, or \textit{de Casteljau's algorithm}
(see, e.g., \cite[\S5.1]{Gol03}) which has the complexity $O(n^2)$.
\end{remark}

Now, we summarise the whole idea in Algorithm \ref{A:Alg}.

\begin{algorithm}\label{A:Alg}
	\ \\[0.5ex]
	\noindent \texttt{Input}: $f$, $k$, $l$, $N$, $a_i$ $(i=0,1,\ldots,k-1)$, $b_j$ $(j=0,1,\ldots,l-1)$\\[0.5ex]
\noindent \texttt{Output}: the coefficients $p_{i,N}$ $(i=0,1,\ldots,N)$ of $w_N$ (see \eqref{E:wNtarget})
\begin{description}
\itemsep2pt
\item[\texttt{Step I.}] Compute the coefficients $p_{i,m-1}$ $(i=0,1,\ldots,m-1)$ of $w_{m-1}$ by \eqref{E:p1} and \eqref{E:p2}
for $n := m-1$, where $m:=k+l$.
\item[\texttt{Step II.}] For $n=m,m+1,\ldots,N$,
\begin{enumerate}
\itemsep2pt
\item compute the outer coefficients $p_{i,n}$ $(i=0,1,\ldots,k-1)$ and $p_{n-j,n}$ $(j=0,1,\ldots,l-1)$ of $w_n$ by \eqref{E:p1} and \eqref{E:p2}, respectively;
\item compute $c_{i,j}^{(n-m)}$ $(i,j=0,1,\ldots,n-m)$ using Lemma \ref{L:c-tab};
\item compute the coefficients $p_{j,n-1}^{(r)}$ $(j=0,1,\ldots,n-r-1)$ of $w_{n-1}^{(r)}$ $(r=1,2,\ldots,m-1)$ using Remark \ref{R:der2};
\item compute $I_{q,n}$ $(q=0,1,\ldots,n-m)$ by \eqref{E:int} (see Remarks \ref{R:der1} and \ref{R:der2});
\item compute the inner coefficients $p_{i,n}$ $(i=k,k+1,\ldots,n-l)$ of $w_n$ by solving the Toeplitz system of linear equations \eqref{E:sys}
(see Remark \ref{R:Toe}).
\end{enumerate}
\item[\texttt{Step III.}] Return the coefficients $p_{i,N}$ $(i=0,1,\ldots,N)$ of $w_N$.
\end{description}
\end{algorithm}

%%%%%%%%%%%%%%%%%%%%%%%%%%%%%%%%%%%%%%%%%%%%%%%%%%%%%%%%%%%%%%%%%%%%%%%%%%%%%%
%%%%%%%%%%%%%%%%%%%%%%%%%%%%%%%%%%%%%%%%%%%%%%%%%%%%%%%%%%%%%%%%%%%%%%%%%%%%%%
\section{Examples}\label{S:Ex}
%%%%%%%%%%%%%%%%%%%%%%%%%%%%%%%%%%%%%%%%%%%%%%%%%%%%%%%%%%%%%%%%%%%%%%%%%%%%%%
%%%%%%%%%%%%%%%%%%%%%%%%%%%%%%%%%%%%%%%%%%%%%%%%%%%%%%%%%%%%%%%%%%%%%%%%%%%%%%

Results of the experiments were obtained in Maple{\small \texttrademark}13 using $32$-digit arithmetic. Integrals \eqref{E:int} were computed using the
Maple{\small \texttrademark}13 function \texttt{int} with the option \texttt{numeric}. Further on in this section, we will use the following notation:
\begin{equation}\label{E:errfun}
\varepsilon_n(x) := |y(x)-w_n(x)| \qquad (0\le x\le1)
\end{equation}
is the \textit{error function}, and
\begin{equation}\label{E:maxErr}
E_n := \max_{x \in Q_M} \varepsilon_n(x) \approx \max_{x \in [0,1]} \varepsilon_n(x)
\end{equation}
is the \textit{maximum error}, where $Q_M := \left\{0, 1/M, 2/M,\ldots, 1\right\}$ with $M=200$.

\begin{example}\label{Ex:1}
First, let us consider the following second order nonlinear differential equation:
\begin{equation*}
y''(x) = [y'(x)]^2+1 \qquad  (0 \le x \le 1)
\end{equation*}
with the boundary conditions
\begin{equation*}
y(0) = 0, \quad y(1) = 0.
\end{equation*}
According to \cite[chapter II, \S4]{GGL85},
\begin{equation*}
y(x) = -\ln{\frac{\cos\left(x-\frac{1}{2}\right)}{\cos\frac{1}{2}}}
\end{equation*}
is the exact solution. We have computed the approximate solutions $w_n$ $(n=2,3,\ldots,20)$. Fig.~\ref{figure1}
illustrates the error functions $\varepsilon_3$,  $\varepsilon_8$ and $\varepsilon_{20}$. For the list of maximum errors
\eqref{E:maxErr}, see Table \ref{tab:table1}.

\begin{figure}[H]
\captionsetup{margin=0pt, font={scriptsize}}
\begin{center}
\setlength{\tabcolsep}{0mm}
\begin{tabular}{c}
\subfloat[]{\label{figure1a}\includegraphics[width=0.33\textwidth]{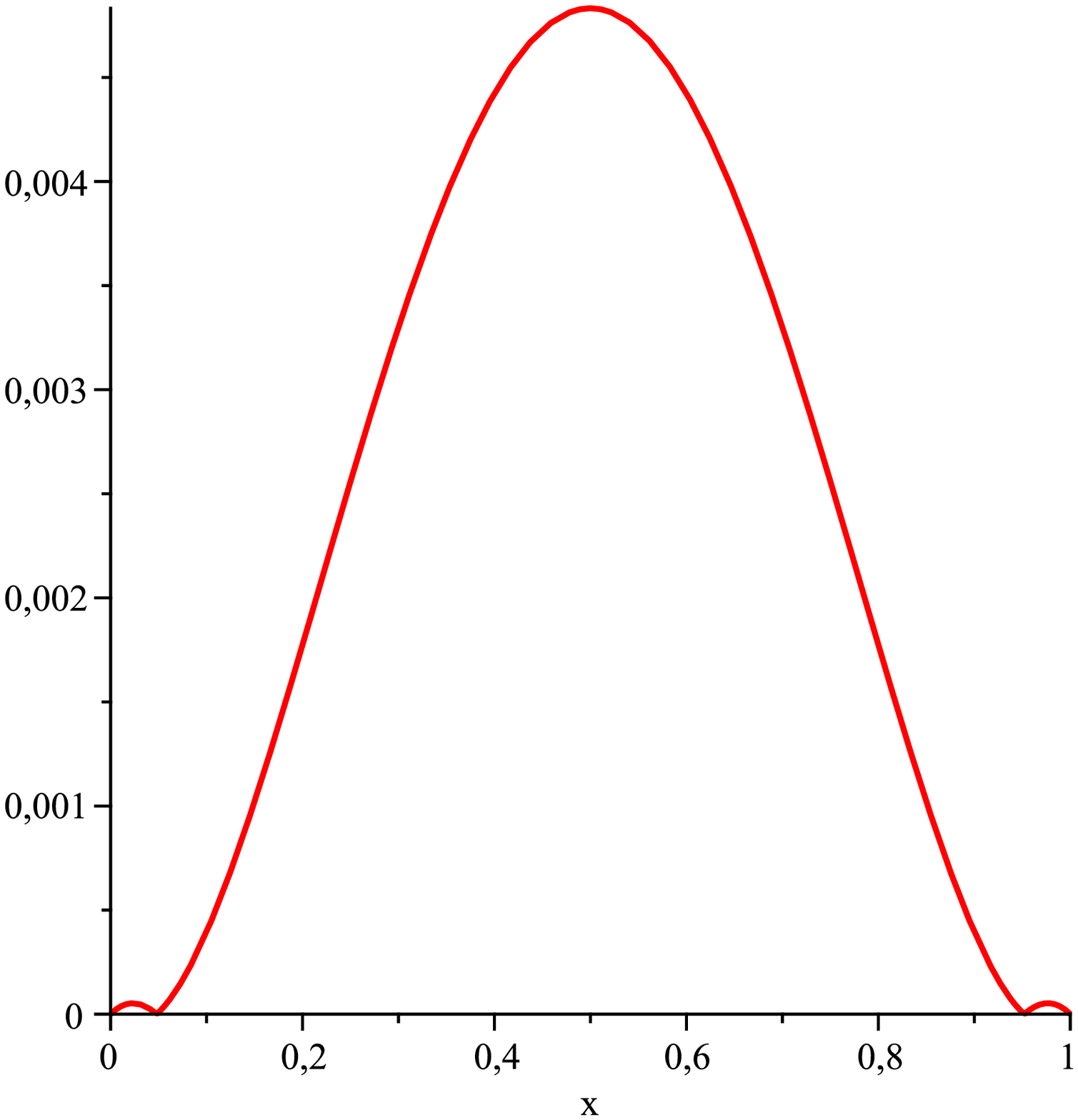}}
\subfloat[]{\label{figure1b}\includegraphics[width=0.33\textwidth]{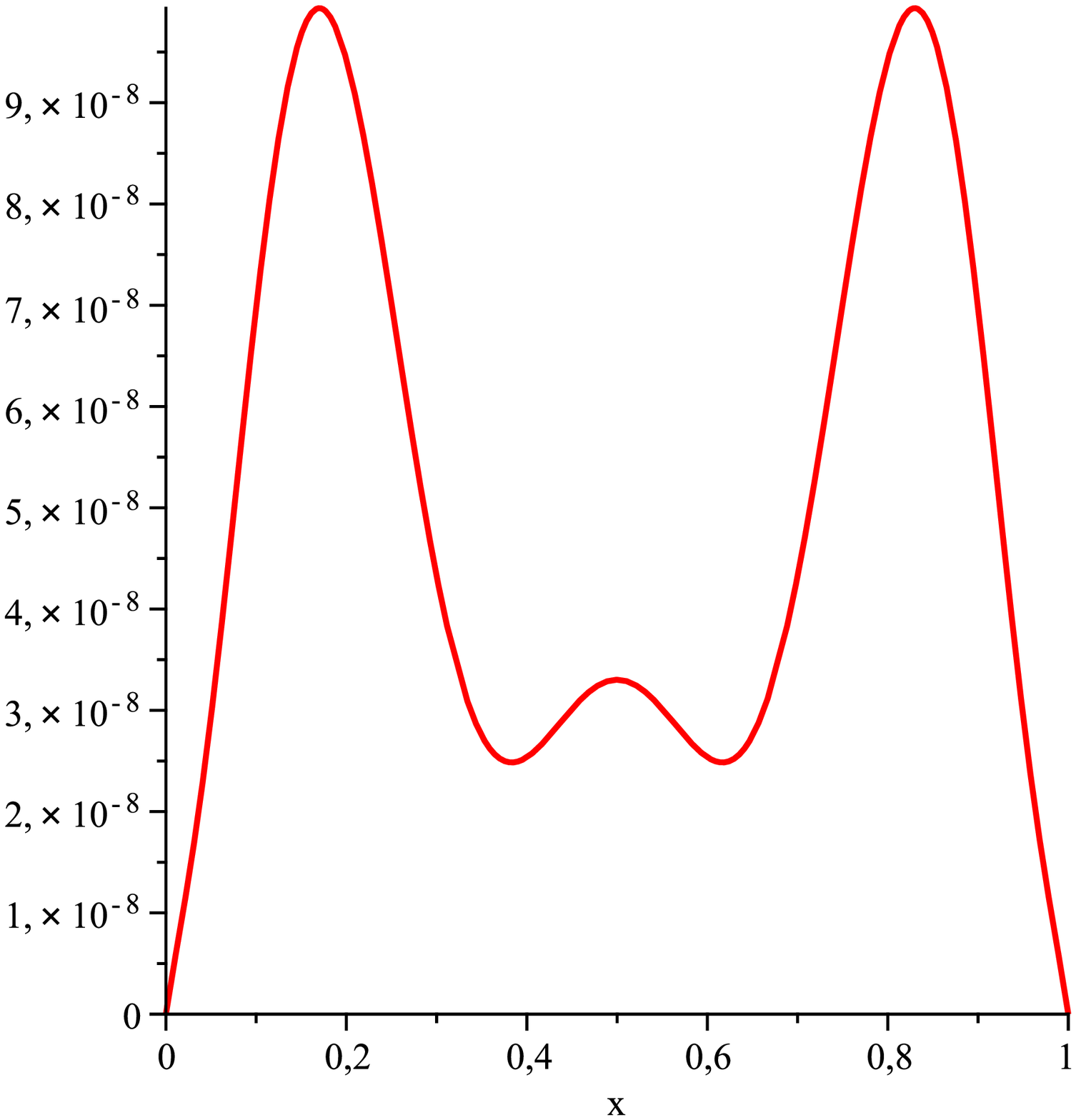}}
\subfloat[]{\label{figure1c}\includegraphics[width=0.33\textwidth]{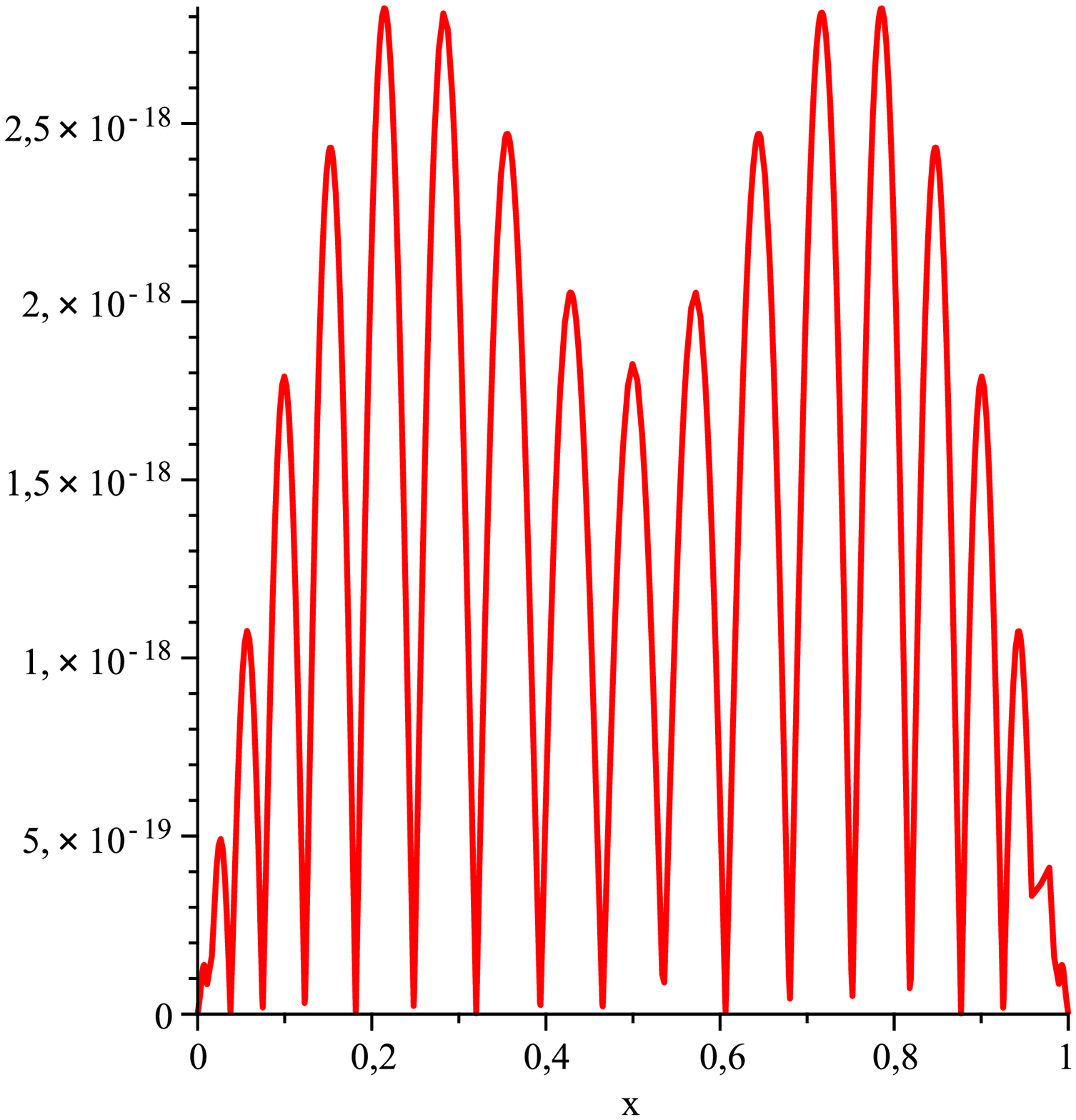}}
\end{tabular}
    \caption{Plots of (a) $\varepsilon_3$, (b) $\varepsilon_8$, and (c) $\varepsilon_{20}$ for Example \ref{Ex:1}.}\label{figure1}
\end{center}
\end{figure}
\end{example}

\begin{example}\label{Ex:2}
Now, let us solve the following fourth order linear differential equation:
\begin{equation*}
y^{(4)}(x) = -2y''(x) - y(x) \qquad  (0 \le x \le 1)
\end{equation*}
with the boundary conditions
\begin{align*}
&y(0) = 3, \quad y'(0) = 3,\\
&y(1) = 0, \quad y'(1) = 0.
\end{align*}
One can check that the exact solution is
\begin{equation*}
y(x) = \frac{3}{2}\sec^2(1)[(4- 3x)\sin(x) -x\sin(2-x) - (3x-1)\cos(x) + (x+1)\cos(2-x)]
\end{equation*}
(cf.~\cite[\S4.3]{ZC08}). The maximum errors for our approximate solutions $w_n$ $(n=4,5,\ldots,20)$ are shown in Table \ref{tab:table1}.
Plots of the selected error functions \eqref{E:errfun} can be seen in Fig.~\ref{figure2}.

\begin{figure}[H]
\captionsetup{margin=0pt, font={scriptsize}}
\begin{center}
\setlength{\tabcolsep}{0mm}
\begin{tabular}{c}
\subfloat[]{\label{figure2a}\includegraphics[width=0.33\textwidth]{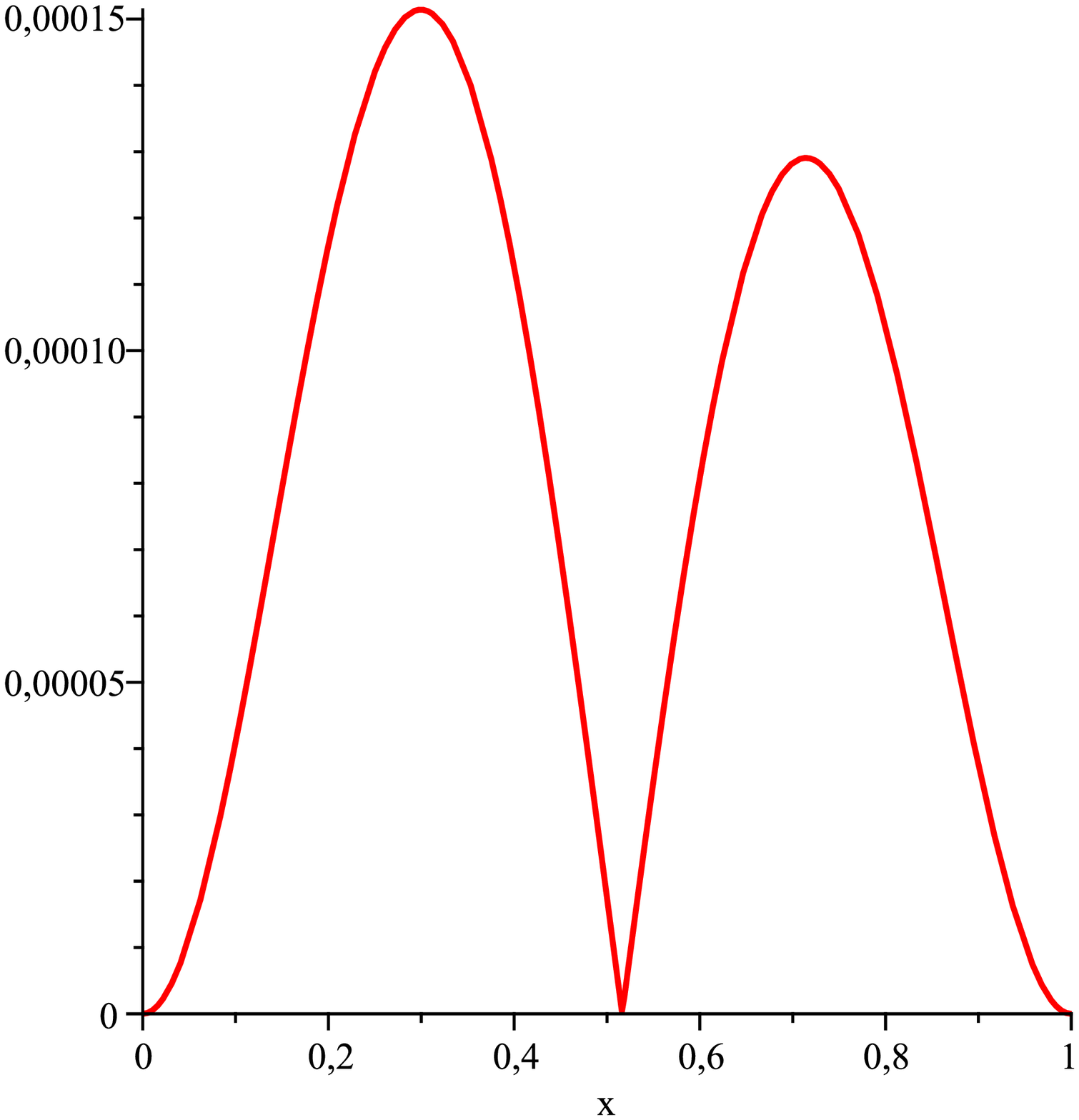}}
\subfloat[]{\label{figure2b}\includegraphics[width=0.33\textwidth]{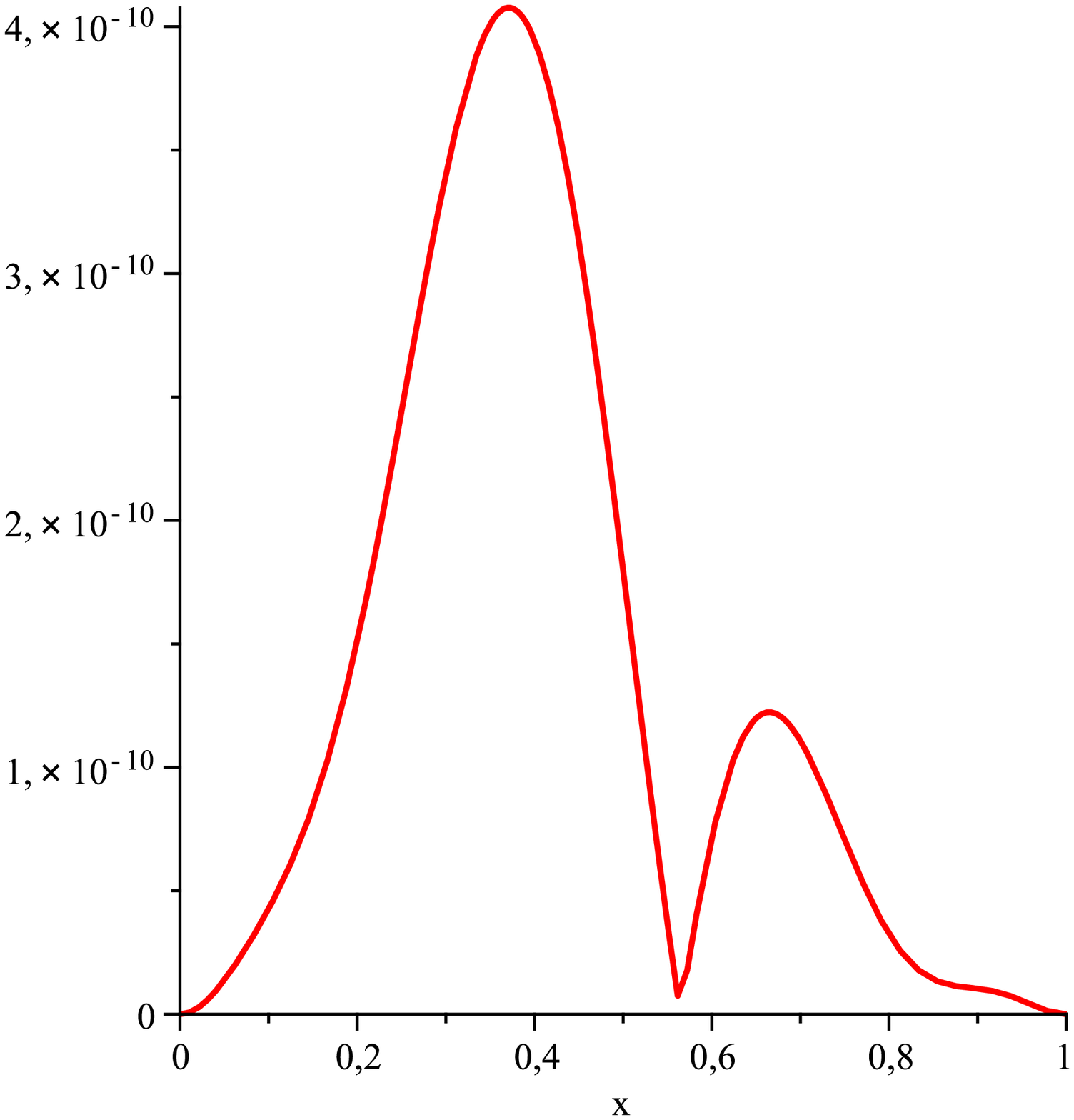}}
\subfloat[]{\label{figure2c}\includegraphics[width=0.33\textwidth]{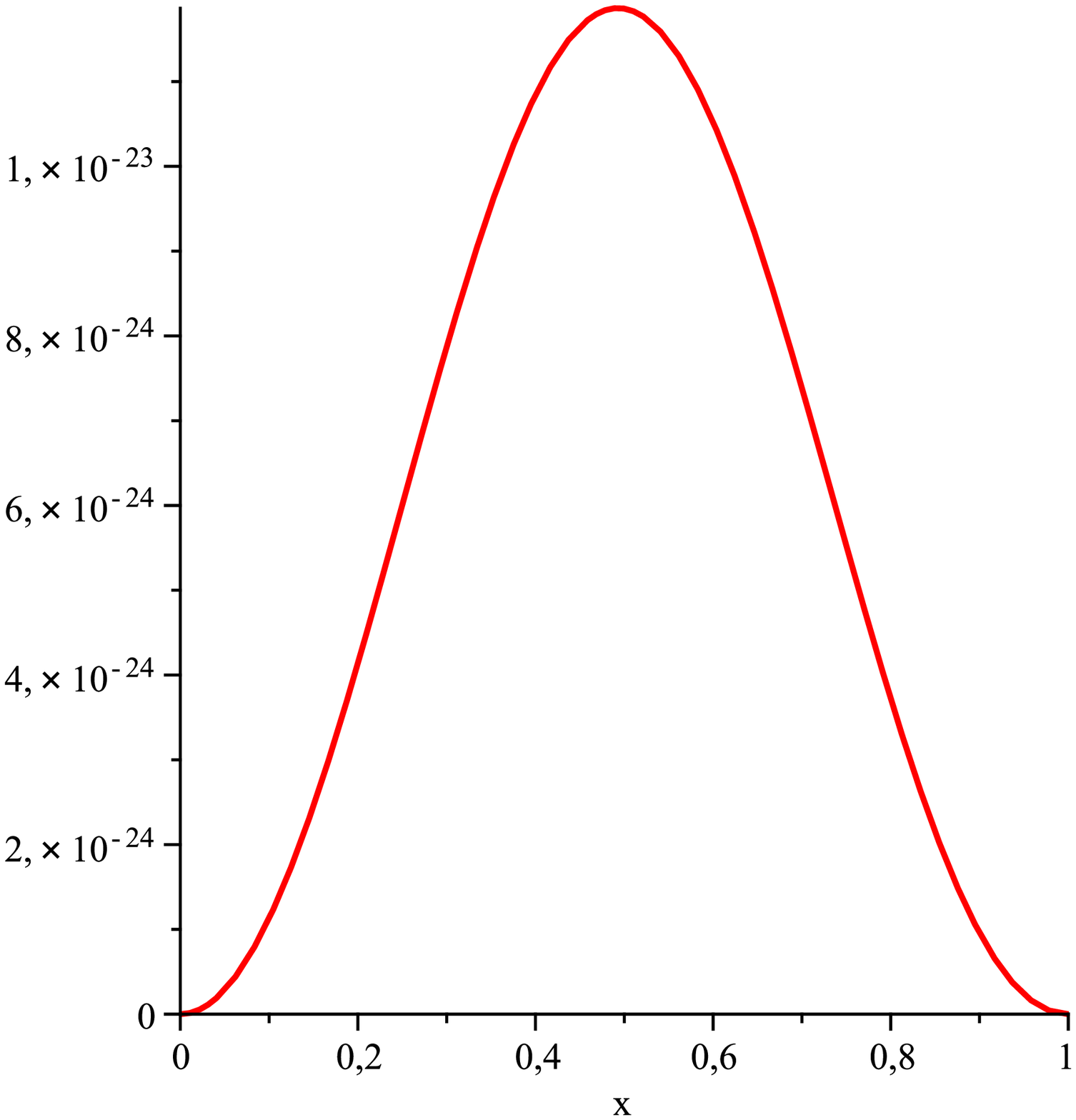}}
\end{tabular}
    \caption{Plots of (a) $\varepsilon_6$, (b) $\varepsilon_{10}$, and (c) $\varepsilon_{20}$ for Example \ref{Ex:2}.}\label{figure2}
\end{center}
\end{figure}
\end{example}

\begin{example}\label{Ex:3}
The following fourth order nonlinear differential equation:
\begin{equation*}
y^{(4)}(x) = \frac{[y'''(x)]^2}{y''(x)} \qquad  (0 \le x \le 1)
\end{equation*}
with the conditions
\begin{equation*}
y(0) = 2, \quad y'(0) = -1, \quad y''(0) = 3, \quad y'''(0) = 1
\end{equation*}
has the exact solution
\begin{equation*}
y(x) = -25 -10x + 27e^{\frac{x}{3}}
\end{equation*}
(cf.~\cite[\S4.2.1.18]{PV95}).
Let us consider the approximate solutions $w_n$ $(n=4,5,\ldots,20)$. Selected error functions \eqref{E:errfun}
are shown in Fig.~\ref{figure3}. The list of maximum errors \eqref{E:maxErr} is given in Table \ref{tab:table1}.

\begin{figure}[H]
\captionsetup{margin=0pt, font={scriptsize}}
\begin{center}
\setlength{\tabcolsep}{0mm}
\begin{tabular}{c}
\subfloat[]{\label{figure3a}\includegraphics[width=0.33\textwidth]{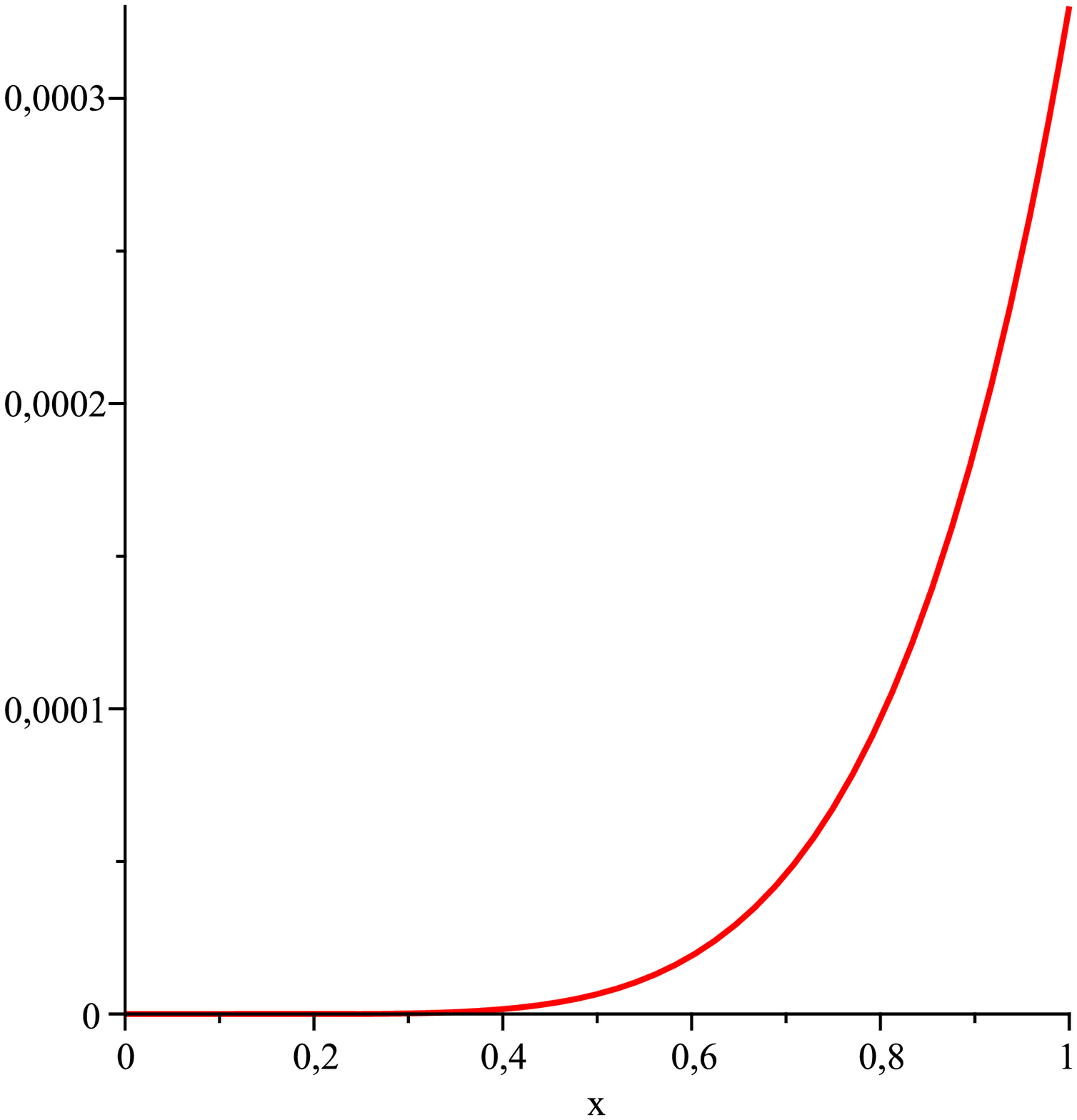}}
\subfloat[]{\label{figure3b}\includegraphics[width=0.33\textwidth]{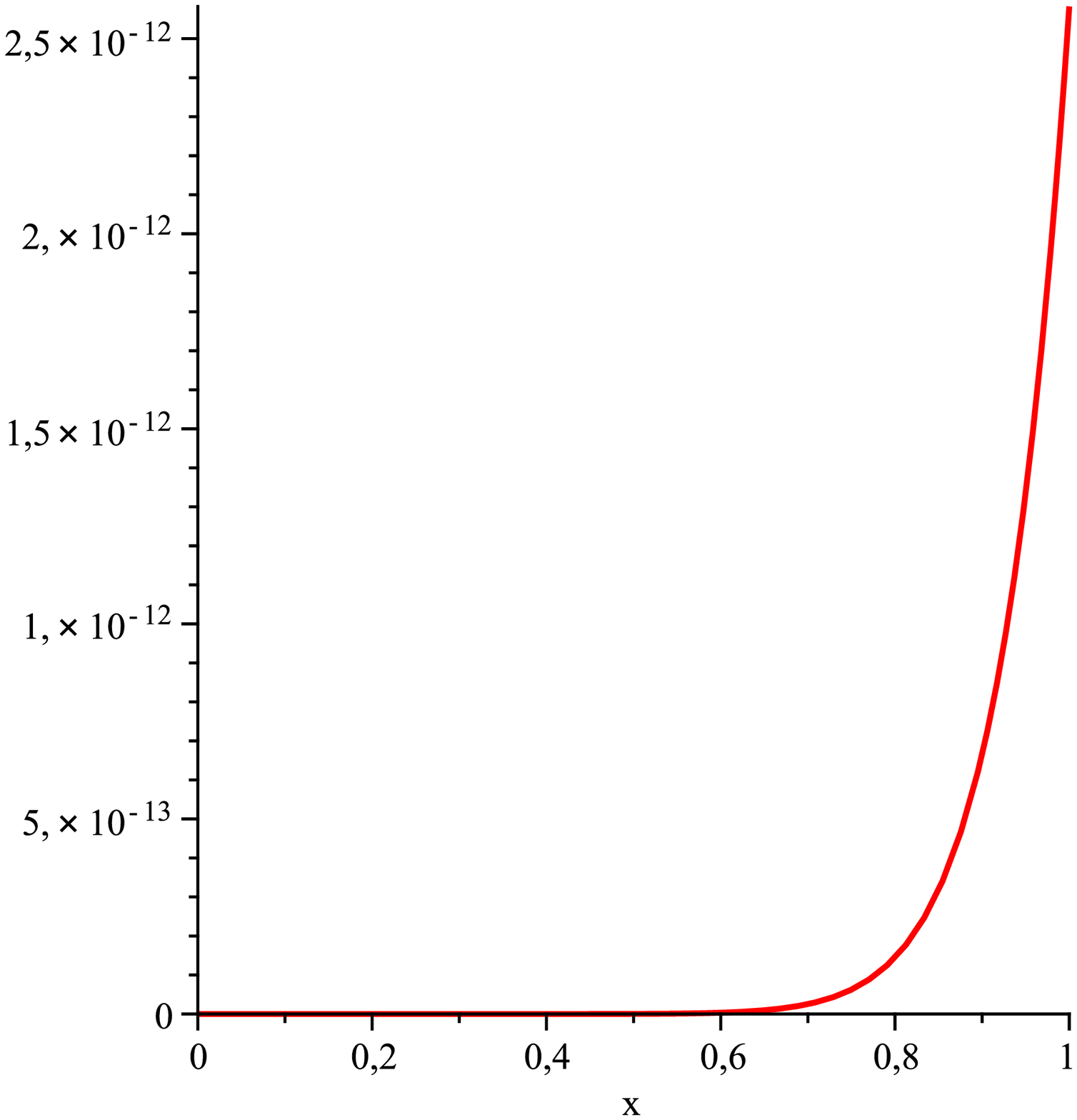}}
\subfloat[]{\label{figure3c}\includegraphics[width=0.33\textwidth]{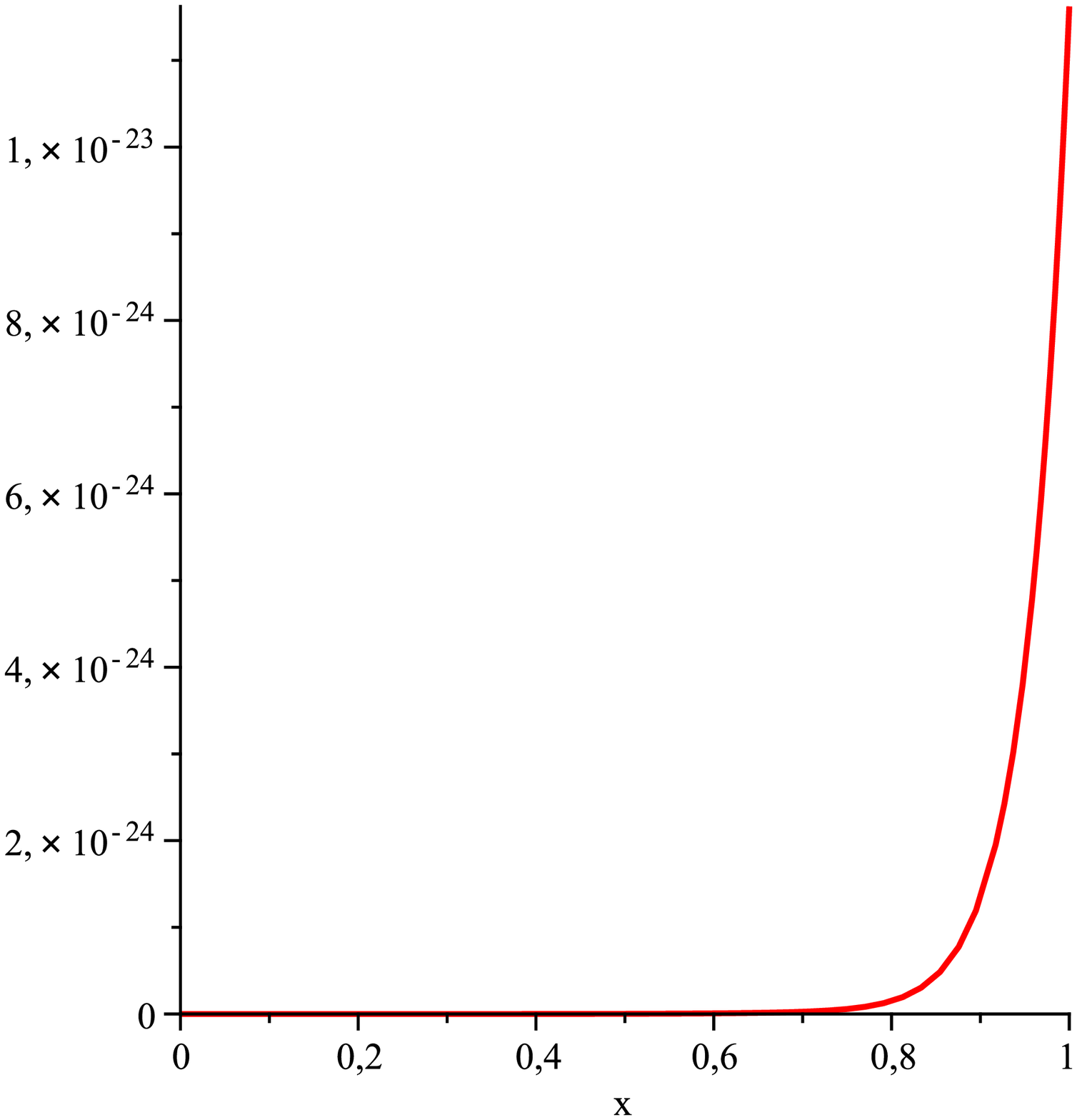}}
\end{tabular}
    \caption{Plots of (a) $\varepsilon_5$, (b) $\varepsilon_{12}$, and (c) $\varepsilon_{20}$ for Example \ref{Ex:3}.}\label{figure3}
\end{center}
\end{figure}
\end{example}

\begin{example}\label{Ex:4}
Next, we apply the algorithm to the following third order linear differential equation:
\begin{equation*}
y'''(x) = 4xy'(x) + 2y(x) \qquad  (0 \le x \le 1)
\end{equation*}
with the conditions
\begin{equation*}
y(0) = 1, \quad y'(0) = 0, \quad y(1) = 0.
\end{equation*}
The exact solution is
\begin{equation*}
y(x) = c_1\Ai^2(x) + c_2\Ai(x)\Bi(x) + c_3\Bi^2(x),
\end{equation*}
where
\begin{equation*}
c_1 := -3\Ai(1)\Bi(1)c_4,\quad
c_2 := \left[3\Ai^2(1)+\Bi^2(1)\right]c_4,\quad
c_3 := - \Ai(1)\Bi(1)c_4
\end{equation*}
with
$$
c_4 := \frac{3^{\frac{5}{6}}\Gamma^2(\frac{2}{3})}{3\Ai^2(1) + \Bi^2(1) -2\sqrt{3}\Ai(1)\Bi(1)}
$$
(cf.~\cite[\S10.4.57]{AS72}). Here $\Ai$ and $\Bi$ are \textit{Airy functions of the first and second kind}, respectively (for details, see, e.g.,
\cite[\S10.4]{AS72}), and $\Gamma$ is the \textit{gamma function} (see, e.g., \cite[\S6.1]{AS72}). Values of the Airy functions were computed using
\texttt{AiryAi} and \texttt{AiryBi} procedures provided by Maple{\small \texttrademark}13. Let us consider the approximate solutions
$w_n$ $(n=3,4,\ldots,20)$. For the plots of the error functions $\varepsilon_6$, $\varepsilon_{10}$ and $\varepsilon_{20}$, see Fig.~\ref{figure4}.
Maximum errors \eqref{E:maxErr} are listed in Table \ref{tab:table1}.

\begin{figure}[H]
\captionsetup{margin=0pt, font={scriptsize}}
\begin{center}
\setlength{\tabcolsep}{0mm}
\begin{tabular}{c}
\subfloat[]{\label{figure4a}\includegraphics[width=0.33\textwidth]{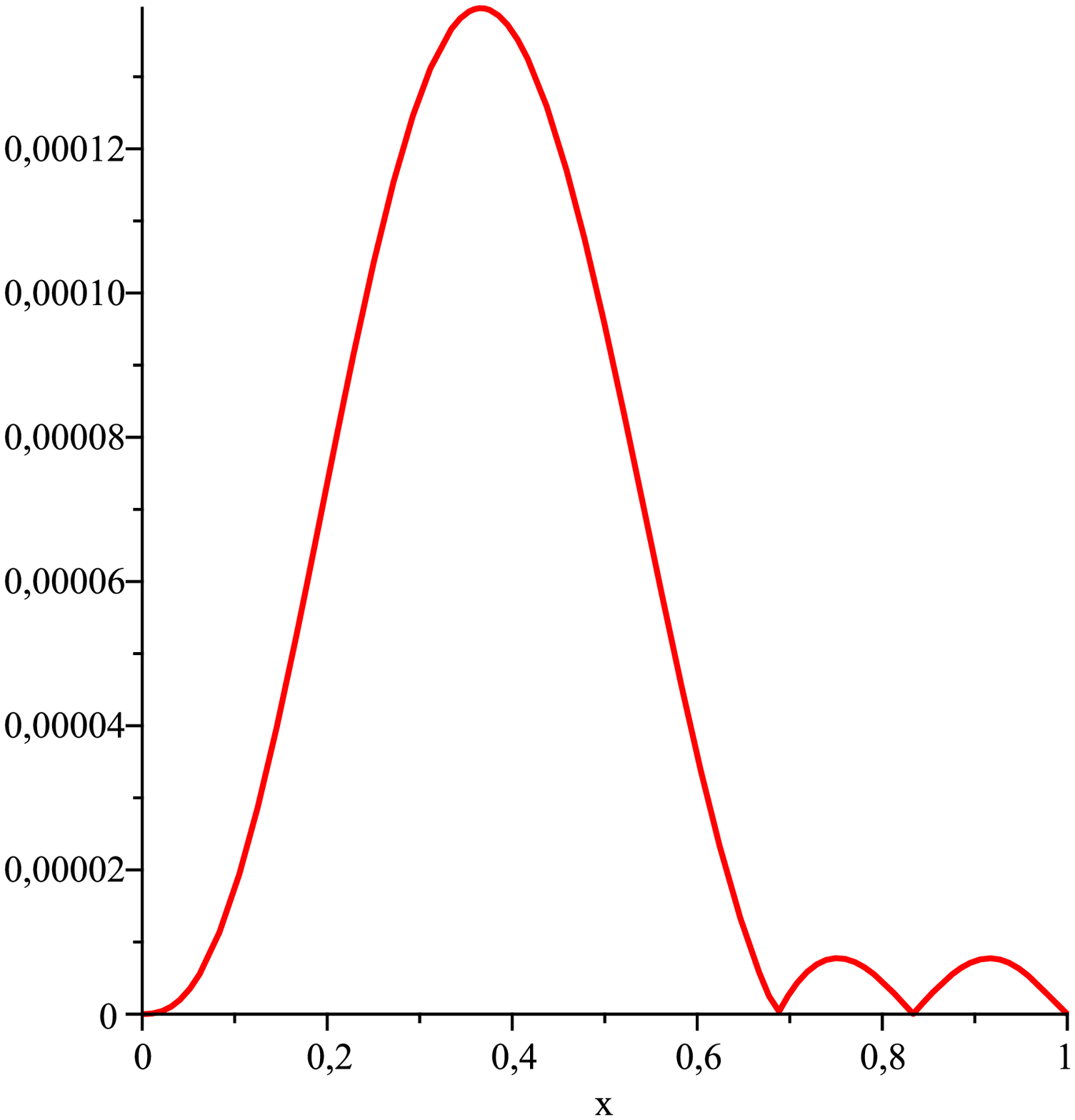}}
\subfloat[]{\label{figure4b}\includegraphics[width=0.33\textwidth]{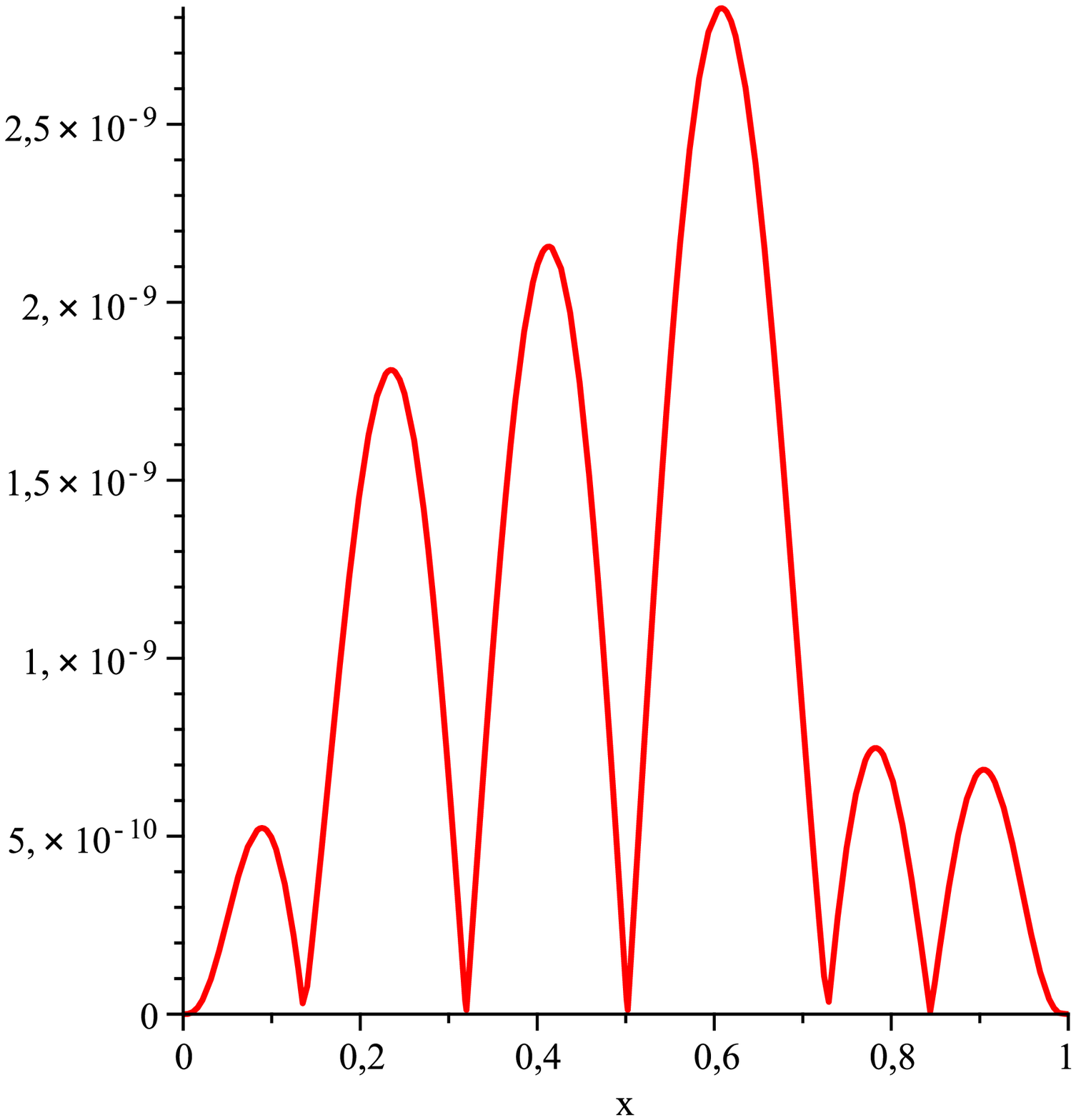}}
\subfloat[]{\label{figure4c}\includegraphics[width=0.33\textwidth]{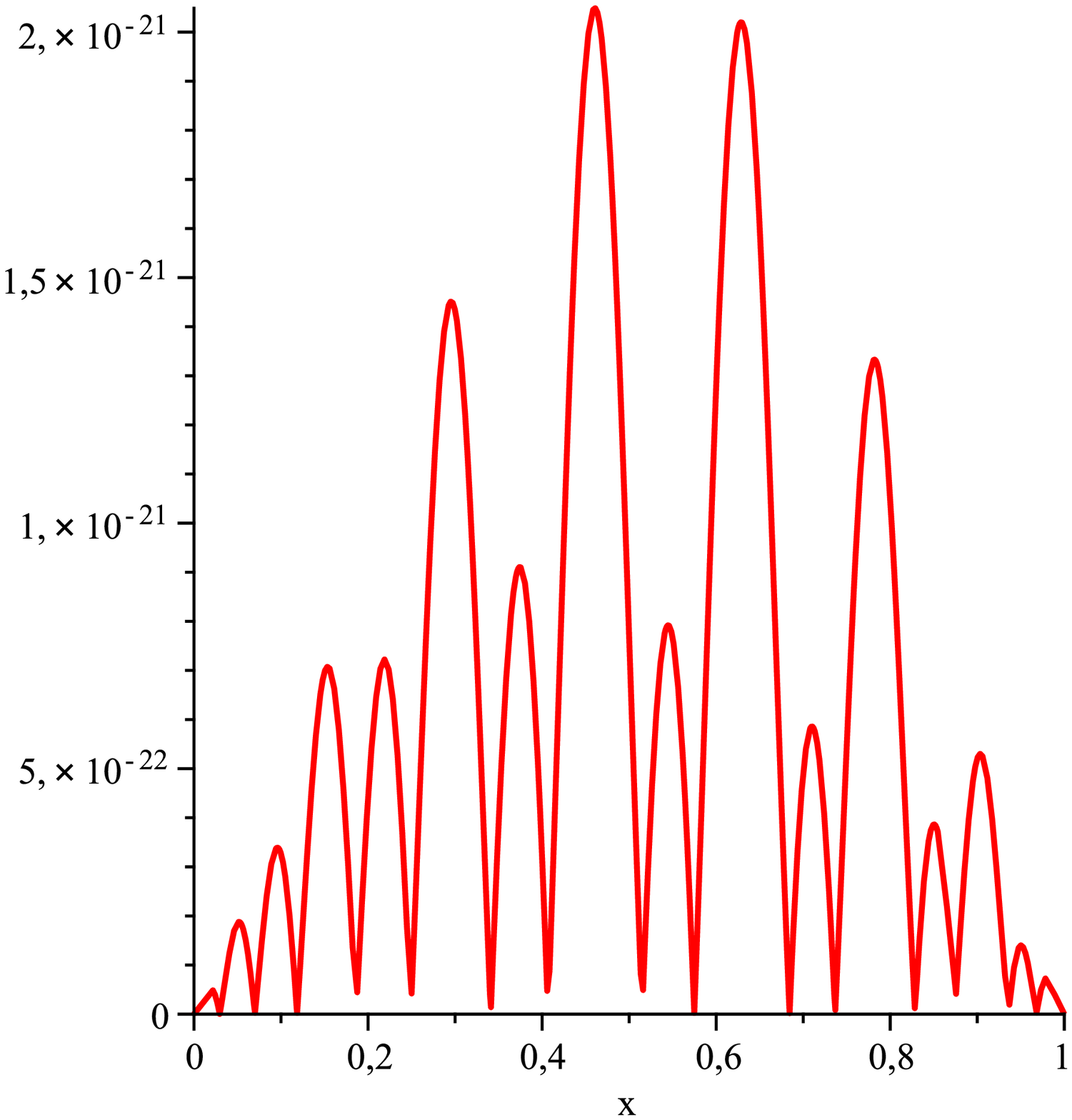}}
\end{tabular}
    \caption{Plots of (a) $\varepsilon_6$, (b) $\varepsilon_{10}$, and (c) $\varepsilon_{20}$ for Example \ref{Ex:4}.}\label{figure4}
\end{center}
\end{figure}
\end{example}

\begin{example}\label{Ex:5}
Finally, let us consider the following second order linear differential equation:
\begin{equation*}
y''(x) = -(x+2)^2y(x) \qquad  (0 \le x \le 1)
\end{equation*}
subject to the boundary conditions
\begin{equation*}
y(0) = \sqrt{2}\left[J_{\frac{1}{4}}(2)+Y_{\frac{1}{4}}(2)\right],\quad
y'(0) = 2\sqrt{2}\left[J_{-\frac{3}{4}}(2)+Y_{-\frac{3}{4}}(2)\right]
\end{equation*}
with the exact solution
\begin{equation*}
y(x) = \sqrt{x+2}\left[J_{\frac{1}{4}}\left(\frac{1}{2}(x+2)^2\right) + Y_{\frac{1}{4}}\left(\frac{1}{2}(x+2)^2\right)\right]
\end{equation*}
(cf.~\cite[\S8.49]{GR07}).
Here $J_{\frac{1}{4}}$ and $Y_{\frac{1}{4}}$ are \textit{Bessel functions of order $\frac{1}{4}$ of the first and second kind},
respectively (for details, see, e.g., \cite[\S2.1.2.121]{PV95}). Values of the Bessel functions were computed using \texttt{BesselJ} and \texttt{BesselY}
procedures provided by Maple{\small \texttrademark}13. Maximum errors \eqref{E:maxErr}, and plots of the selected error functions \eqref{E:errfun} for our
approximate solutions $w_n$ $(n=2,3,\ldots,20)$ can be found in Table \ref{tab:table1} and Fig.~\ref{figure5}, respectively.

\begin{figure}[H]
\captionsetup{margin=0pt, font={scriptsize}}
\begin{center}
\setlength{\tabcolsep}{0mm}
\begin{tabular}{c}
\subfloat[]{\label{figure5a}\includegraphics[width=0.33\textwidth]{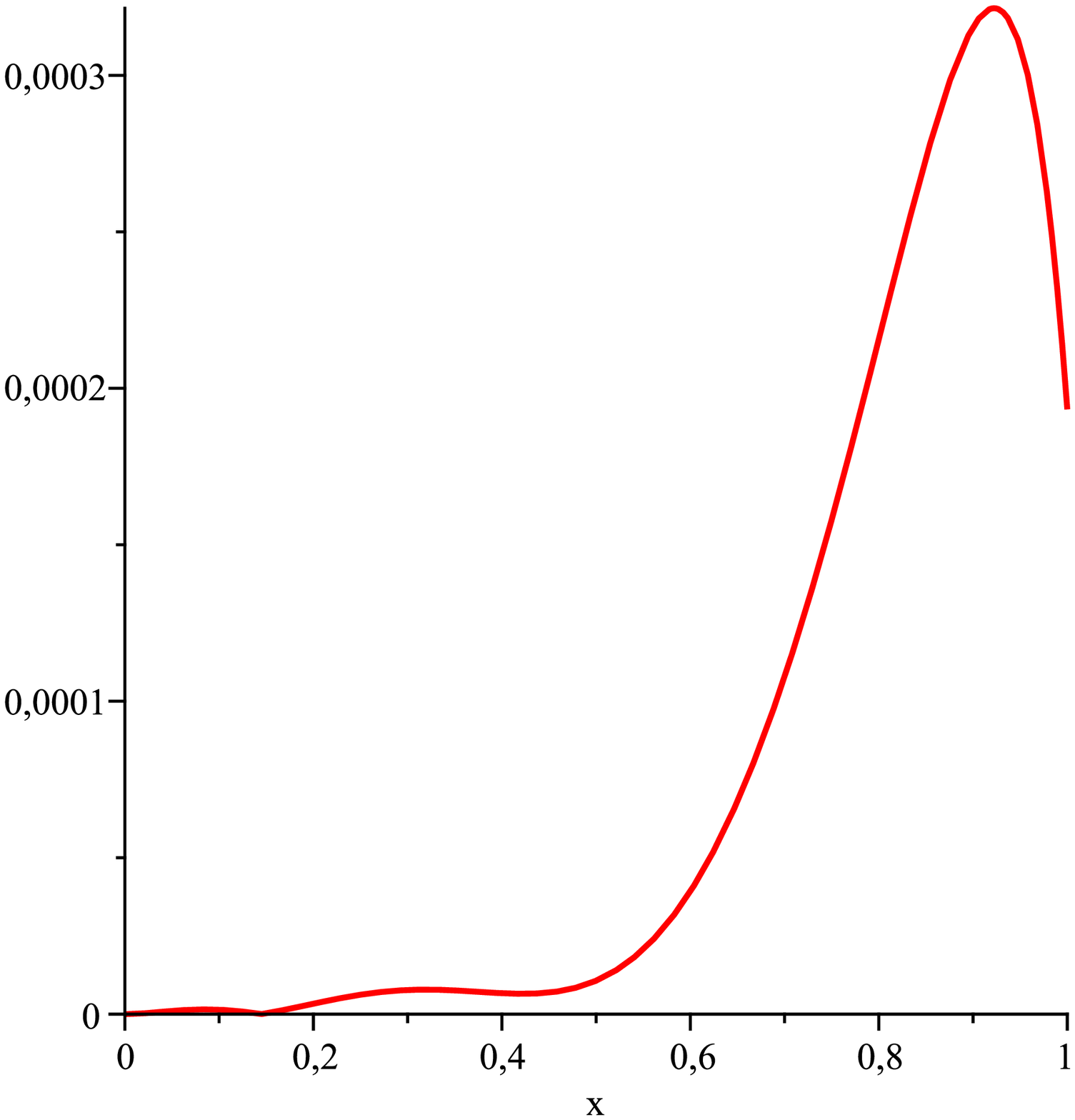}}
\subfloat[]{\label{figure5b}\includegraphics[width=0.33\textwidth]{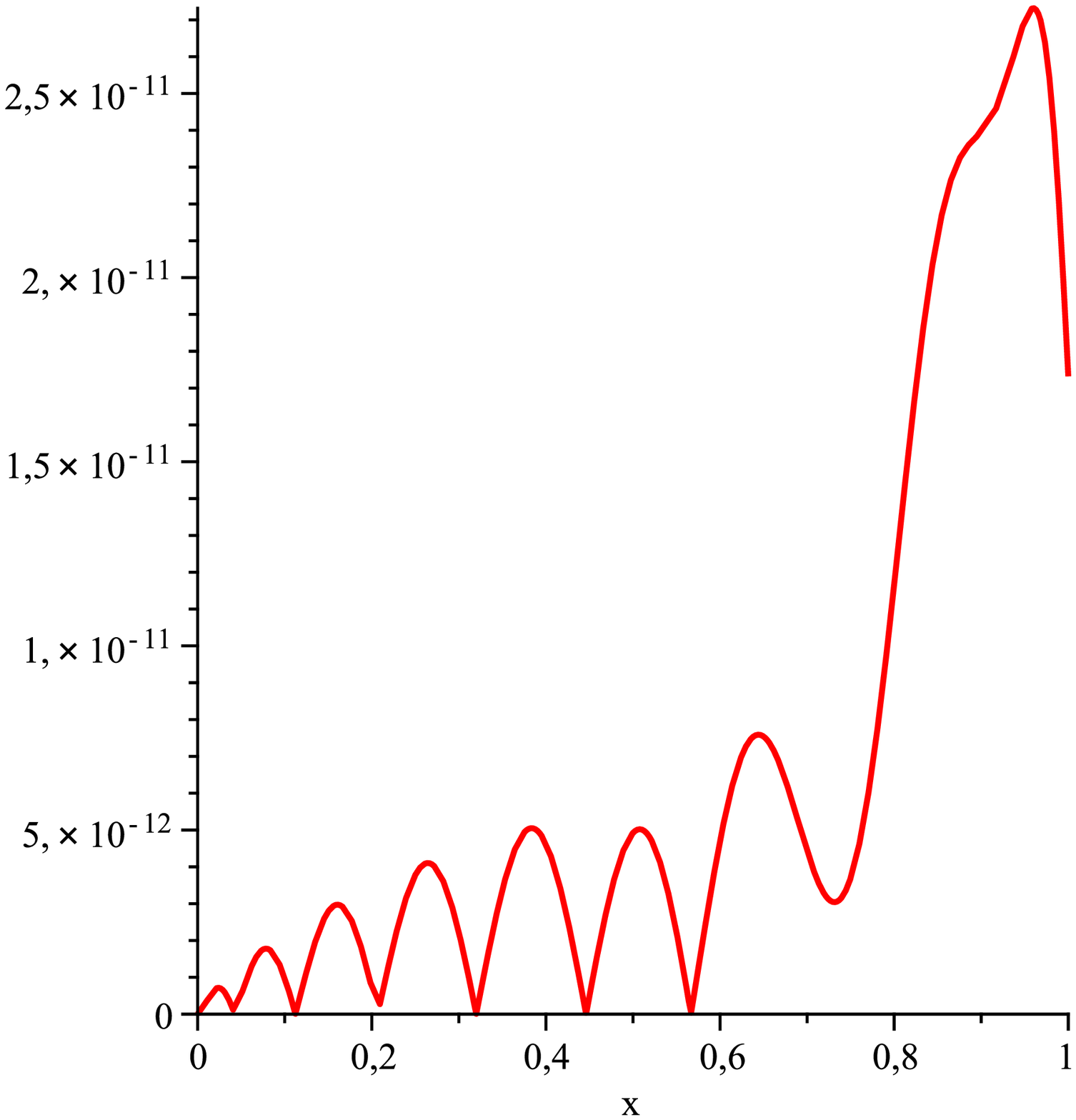}}
\subfloat[]{\label{figure5c}\includegraphics[width=0.33\textwidth]{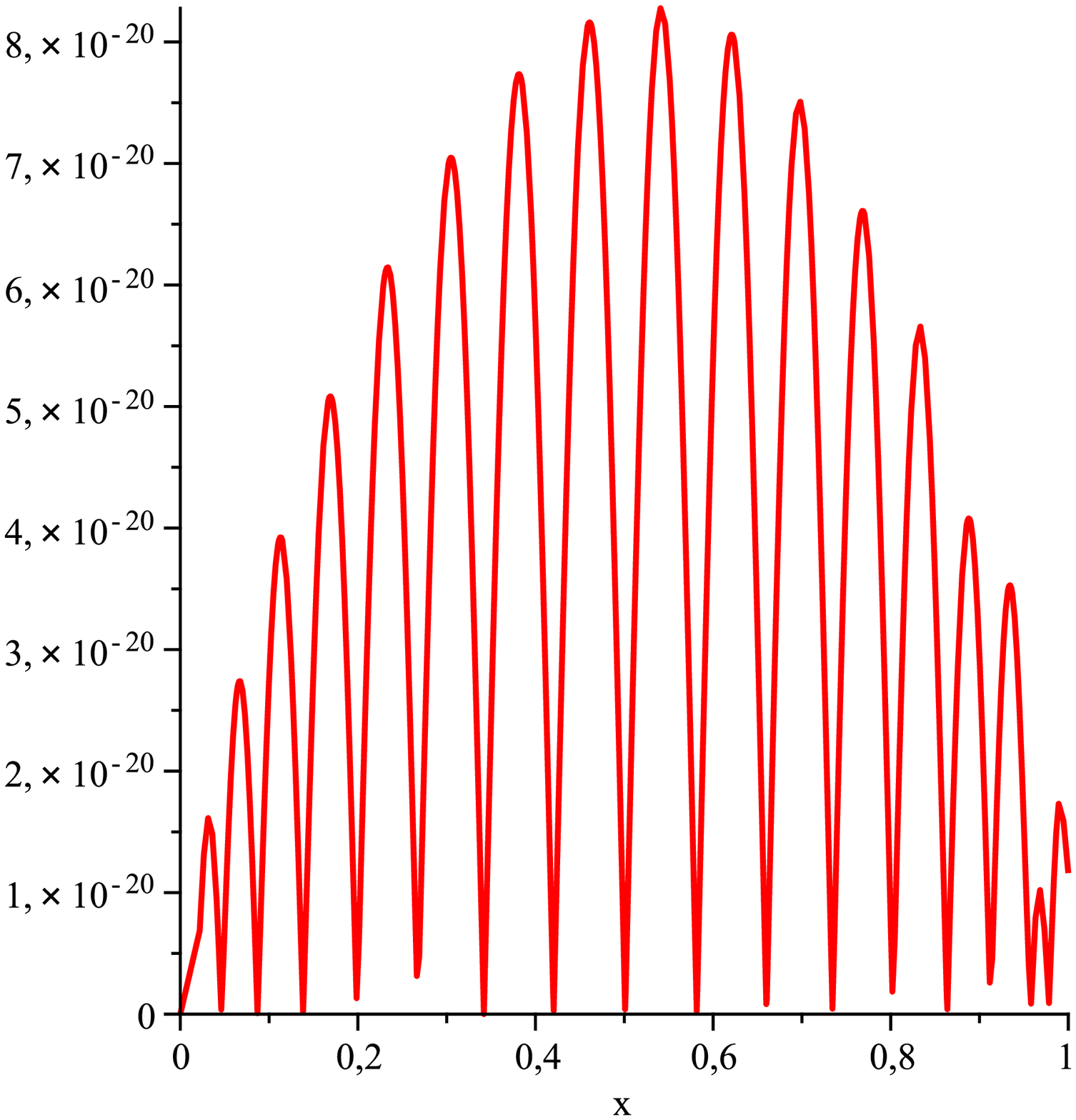}}
\end{tabular}
    \caption{Plots of (a) $\varepsilon_7$, (b) $\varepsilon_{13}$, and (c) $\varepsilon_{20}$ for Example \ref{Ex:5}.}\label{figure5}
\end{center}
\end{figure}
\end{example}

\begin{table}[H]
\captionsetup{margin=0pt, font={scriptsize}}
\ra{1.3}
\begin{center}
  \begin{tabular}{lccccc}
    \hline
    $n$ & Example \ref{Ex:1} & Example \ref{Ex:2} & Example \ref{Ex:3} & Example \ref{Ex:4} & Example \ref{Ex:5}\\
     \hline
    $2$ & $5.58e{-}3$      & -----           & -----           & -----           & $1.48e{+}0$ \\
    $3$ & $4.83e{-}3$      & -----           & -----           & $3.40e{-}2$     & $5.56e{-}1$\\
    $4$ & $5.28e{-}4$      & $8.11e{-}3$     & $2.88e{-}3$     & $1.03e{-}2$     & $1.94e{-}1$\\
    $5$ & $7.90e{-}5$      & $4.32e{-}4$     & $3.30e{-}4$     & $1.64e{-}3$     & $9.60e{-}2$\\
    $6$ & $4.98e{-}6$      & $1.51e{-}4$     & $3.30e{-}5$     & $1.40e{-}4$     & $9.18e{-}3$\\
    $7$ & $1.56e{-}6$      & $4.21e{-}6$     & $2.85e{-}6$     & $6.81e{-}6$     & $3.21e{-}4$\\
    $8$ & $9.93e{-}8$      & $3.55e{-}7$     & $2.17e{-}7$     & $5.88e{-}7$     & $1.06e{-}4$\\
    $9$ & $2.05e{-}8$      & $9.85e{-}9$     & $1.47e{-}8$     & $4.44e{-}8$     & $1.15e{-}5$\\
    $10$ & $1.19e{-}9$     & $4.08e{-}10$    & $9.01e{-}10$    & $2.83e{-}9$     & $8.50e{-}7$\\
    $11$ & $4.56e{-}10$    & $1.29e{-}11$    & $5.03e{-}11$    & $1.89e{-}10$    & $4.59e{-}8$\\
    $12$ & $1.27e{-}11$    & $5.34e{-}13$    & $2.58e{-}12$    & $1.78e{-}11$    & $1.52e{-}9$\\
    $13$ & $9.58e{-}12$    & $2.21e{-}14$    & $1.23e{-}13$    & $9.10e{-}13$    & $2.73e{-}11$\\
    $14$ & $2.82e{-}13$    & $1.04e{-}15$    & $5.42e{-}15$    & $5.82e{-}14$    & $5.76e{-}12$\\
    $15$ & $2.14e{-}13$    & $4.97e{-}17$    & $2.24e{-}16$    & $4.63e{-}15$    & $3.96e{-}13$\\
    $16$ & $5.69e{-}15$    & $2.41e{-}18$    & $8.71e{-}18$    & $2.18e{-}16$    & $1.65e{-}14$\\
    $17$ & $5.00e{-}15$    & $1.18e{-}19$    & $3.19e{-}19$    & $1.23e{-}17$    & $4.59e{-}16$\\
    $18$ & $1.24e{-}16$    & $5.73e{-}21$    & $1.11e{-}20$    & $8.66e{-}19$    & $1.42e{-}17$\\
    $19$ & $1.19e{-}16$    & $2.79e{-}22$    & $3.64e{-}22$    & $3.95e{-}20$    & $3.45e{-}19$\\
    $20$ & $2.82e{-}18$    & $1.19e{-}23$    & $1.16e{-}23$    & $2.05e{-}21$    & $8.27e{-}20$\\
    \hline
  \end{tabular}
\end{center}
\caption{Maximum errors $E_n$ of the approximate solutions $w_n$ of the problems from Examples \ref{Ex:1}--\ref{Ex:5}.}
\label{tab:table1}
\end{table}

%%%%%%%%%%%%%%%%%%%%%%%%%%%%%%%%%%%%%%%%%%%%%%%%%%%%%%%%%%%%%%%%%%%%%%%%%%%%%%
%%%%%%%%%%%%%%%%%%%%%%%%%%%%%%%%%%%%%%%%%%%%%%%%%%%%%%%%%%%%%%%%%%%%%%%%%%%%%%
\section{Conclusions and future work}\label{S:Co}
%%%%%%%%%%%%%%%%%%%%%%%%%%%%%%%%%%%%%%%%%%%%%%%%%%%%%%%%%%%%%%%%%%%%%%%%%%%%%%
%%%%%%%%%%%%%%%%%%%%%%%%%%%%%%%%%%%%%%%%%%%%%%%%%%%%%%%%%%%%%%%%%%%%%%%%%%%%%%

We presented a new iterative approximate method of solving boundary value problems. In each iteration, the method computes efficiently an approximate polynomial solution in the Bernstein form using least squares approximation combined with some properties of dual Bernstein polynomials.
Both linear and nonlinear differential equations of any order can be solved. Examples confirmed that the method is versatile.
Some of them showed that the method can handle differential equations having quite complicated exact solutions written in terms of special functions,
such as Airy and Bessel functions.

In recent years, there has been a progress in the area of dual bases in general (see, e.g., \cite{Ker13,Woz13,Woz14}). Notice that
our idea could be generalized to work for any pair of dual bases, not only Bernstein and dual Bernstein bases. However, the algorithm would be much more
complicated. For example, dealing with the boundary conditions would be much more difficult, since we would not be able to use Lemma \ref{L:boundary}.
In the near future, we intend to study this generalization.
\bibliographystyle{plain}

%%%%%%%%%%%%%%%%%%%%%%%%%%%%%%%%%%%%%%%%%%%%%%%%%%%%%%%%%%%%%%%%%%%%%%%%%%%%%%
%%%%%%%%%%%%%%%%%%%%%%%%%%%%%%%%%%%%%%%%%%%%%%%%%%%%%%%%%%%%%%%%%%%%%%%%%%%%%%

\end{document}